\documentclass[12pt]{amsart}
\usepackage{amssymb}
\usepackage{amsmath}
\usepackage{mathabx}
\usepackage[latin1]{inputenc}
\usepackage{bbm}
\usepackage{dsfont}
\usepackage{enumerate}
\usepackage{mathtools,mathrsfs}

\usepackage{color}

\usepackage[hyperindex=true, colorlinks=false]{hyperref}
\usepackage{url}

\newcommand{\begla}{\begin{equation}}
\newcommand{\beglab}[1]{\begin{equation}	\label{#1}}
\newcommand{\edla}{\end{equation}}

\newcommand{\eg}{{\it e.g.}\ }

\newcommand{\lhs}{{left-hand side\ }}
\newcommand{\rhs}{{right-hand side\ }}

\newcommand{\ens}{\enspace}

\newcommand{\ti}{\tilde}

\newcommand{\dd}{\mathrm{d}}

\newcommand{\defeq}{\coloneqq} % \newcommand{\defeq}{:=}
\newcommand{\col}{\colon\thinspace}          %% for a map f \col A \to B

\newcommand{\cc}[1]{_{[#1]}}

\def\th{\theta}

\def\Z{{\mathbb Z}}\def\T{{\mathbb T}}\def\R{{\mathbb R}}

\def\cB{{\mathcal B}}
\def\cT{{\mathcal T}}
\def\cN{{\mathcal N}}

\def\cU{{\mathcal U}}

\def\bz{\bar z}

\def\ga{\gamma}

\newcommand{\ka}{\kappa}
\newcommand{\al}{\alpha}

\def\la{\lambda}
\def\th{\theta}

\def\<{\langle}
\def\>{\rangle}

\theoremstyle{plain}

\newtheorem{defin}{Definition}[section]

\newtheorem*{Cor*}{Corollary}
\newtheorem*{Prop*}{Proposition}

\newtheorem{Main}{Theorem}

\def\ti{\tilde}

\def\pa{\partial}

\def\th{\theta}

\newcommand{\ba}{\overline{A}}

\newcommand{\bM}{\overline{M}}

\newcommand{\cF}{\mathcal{F}}

\def\be{\begin{equation}}
\def\ba{{\begin{align}}}
\def\ea{{\end{align}}}

\def\bm{\begin{matrix}}
\def\em{\end{matrix}}

\def\0{{\mathbf 0}}

\newtheorem{thm}{Theorem}[section]
\newtheorem{lemma}[thm]{Lemma}

\newtheorem{prop}[thm]{Proposition}

\theoremstyle{remark}

\numberwithin{equation}{section}
\def\cB{\mathcal{B}}

\theoremstyle{definition}

\def\ssm{\smallsetminus}

\renewcommand{\setminus}{\ssm}

\newcommand{\dist}{\operatorname{dist}}

\newcommand{\eps}{{\epsilon}}

\newcommand{\om}{{{\omega}}}

\newcommand{\N}{{\mathbb N}}
\newcommand{\Q}{{\mathbb Q}}

\def\B0{{\bold{0}}}

\catcode`\@=12

\def\Empty{}
\newcommand\oplabel[1]{
  \def\OpArg{#1} \ifx \OpArg\Empty {} \else
  	\label{#1}
  \fi}
		
\newcommand{\comm}[1]{}
\newcommand{\comment}[1]{}
\newcommand{\Id}{{\rm Id}}
\newcommand{\ID}{{\rm Id}}
\renewcommand{\cB}{\mathcal B}
\renewcommand{\bz}{\bar{z}}
\newcommand{\bx}{\ov{x}}

\newcommand{\tx}{\til{x}}
\newcommand{\bPhi}{\bar{\Phi}}

\newcommand{\bom}{{\ov{\omega}}}
\newcommand{\bT}{\ov{T}}

\newcommand{\bp}{{\overline{p}}}
\newcommand{\bq}{{\overline{q}}}

\def\abs#1{\left\vert#1\right\vert}
\def\norm#1{\Vert#1\Vert}

\def\tdemi{\tfrac{1}{2}}
\def\ha{\widehat}
\def\til{\widetilde}
\def\ov{\overline}

\def\bbz{{\bf z}}
\def\bbT{{\bf T}}

\newcounter{paraga}[subsection]
\renewcommand{\theparaga}{{\bf\arabic{paraga}.}}
\newcommand{\paraga}{\medskip \addtocounter{paraga}{1}
\noindent{\theparaga\ } }

\newcommand{\beq}{\begin{equation}}
\newcommand{\eeq}{\end{equation}}

\newcommand{\bet}{\beta}

\newcommand{\exn}{^{(n)}}
\newcommand{\exnp}{^{(n+1)}}
\newcommand{\exnz}{^{(0)}}

\newcommand{\epsc}{\eps_{\operatorname{c}}}
\newcommand{\epsf}{\eps_{\operatorname{f}}}
\newcommand{\epsH}{\eps_{\operatorname{H}}}

\newcommand{\NN}{{(\N^N)^*}}

\author{B. Fayad, J.-P. Marco, D. Sauzin}
\title{Attracted by an elliptic fixed point}
\begin{document}

\maketitle{}

\begin{abstract} We give examples of symplectic diffeomorphisms
  of~$\R^{6}$ for which the origin is a non-resonant elliptic fixed
  point which attracts an orbit.
\end{abstract}
\section{Introduction}
Consider a symplectic diffeomorphism % $f$ 
of~$\R^{2n}$ % (endowed with the canonical symplectic form) fixing the origin $0$. 
(for the canonical symplectic form)
with a fixed point at the origin.
We say that the fixed point is elliptic of frequency vector
$\om=(\om_1,\ldots,\om_n)\in\R^n$ if the linear part of the
diffeomorphism at the fixed point is conjugate to the rotation map
$$
S_\om : (\R^2)^n\righttoleftarrow,\qquad
S_{\om}(s_1,\ldots,s_n) \defeq (R_{\om_1}(s_1),\ldots,R_{\om_n}(s_n)).
$$ 
Here, for $\bet \in \R$, $R_\bet$ stands for the rigid rotation around
the origin in~$\R^2$ with rotation number $\bet$.
We say that the frequency vector $\om$ is non-resonant
if for any $k \in \Z^n -\{0\}$ we have $(k,\om) \notin \Z$,
%
% if for any $(k,l) \in \Z^n\times \Z -\{0\}$ we have that
% $(k,\om)+l\neq 0$
% 
where $(\cdot,\cdot)$ stands for the Euclidean scalar product.

%We say that a point $z$ is attracted by $0$ if $\lim_{n\to \infty} f^{n}(z)=0$.

It is easy to construct symplectic diffeomorphisms with orbits
attracted by a resonant elliptic fixed point.  
For instance,
the time-$1$ map of the flow generated by the Hamiltonian function
$H(x,y)=y(x^2+y^2)$ in $\R^2$ has a saddle-node type fixed point, at
which the linear part is zero, which attracts all the points on the negative
part of the $x$-axis.
%\margem{faut-il citer l'article de Kaloshin Mather et Valdinoci ici?}
The situation is much subtler in the non-resonant case. 

The Anosov-Katok construction \cite{AK} of ergodic diffeomorphisms by
successive conjugations of periodic rotations of the disc gives
examples of smooth area preserving diffeomorphisms with non-resonant elliptic fixed points at the origin that are Lyapunov
unstable. The method also yields examples of ergodic symplectomorphisms with non-resonant elliptic fixed points in higher dimensions.

These  constructions obtained by the successive conjugation technique have totally
degenerate fixed points since they are $C^\infty$-tangent to a 
rotation $S_\om$ at the origin.
 
% {\red 
In the non-degenerate case, R. Douady gave examples in \cite{douady}
of Lyapunov unstable elliptic points for smooth symplectic
diffeomorphisms for any $n\geq 2$, for which the Birkhoff normal form
has non-degenerate Hessian at the fixed point but is otherwise
arbitrary.
Prior examples for $n=2$ were obtained in \cite{DLC} (note that by KAM theory, a non-resonant elliptic fixed point of a smooth area preserving surface diffeomorphism that has a non zero Birkhoff normal form is accumulated by invariant quasi-periodic smooth curves (see \cite{moser}). Hence in the one dimensional case, non-degeneracy implies that the point is Lyapunov stable).

In both of the above examples, no orbit distinct from the origin converges to it. Indeed, in the Anosov-Katok examples, a sequence of iterates of the diffeomorphism converges uniformly to Identity, hence every orbit is recurrent and no orbit can converge to the origin, besides the origin itself. As for the non-degenerate examples of Douady and Le Calvez, their 
 Lyapunov instability is deduced
from the existence of a sequence of points that converge to the
fixed point, and whose orbits travel, along a simple resonance, away
from the fixed point. By construction,  these examples do not have a
single orbit besides the origin that converges to it. 
% }

Our goal in this paper is to construct an example of a Lyapunov
unstable fixed point for a Gevrey diffeomorphism with an orbit
converging to it.
Recall that, given a real $\al\ge1$, Gevrey-$\al$ regularity is
defined by the requirement that the partial derivatives exist at all
(multi)orders~$\ell$ and are bounded by
$C M^{\abs\ell} \abs{\ell}!^\al$ for some~$C$ and~$M$
(when $\al=1$, this simply means analyticity);
upon fixing a real $L>0$ which essentially stands for the inverse of
the previous~$M$, % previously mentioned~$M$, 
one can define a Banach algebra
$\big(G^{\al,L}(\R^{2n}), \norm{\,.\,}_{\al,L}\big)$.
%
% consisting of Gevrey-$(\al,L)$ real-valued functions.
%
% The precise definitions and properties of Gevrey functions that we use
% in this paper are recalled in Appendix~\ref{App:Gevrey}.

We set $X\defeq(\R^2)^3$ and denote by $\cU^{\al,L}$ the set of
all Gevrey-$(\al,L)$ symplectic diffeomorphisms of~$X$ which
%\margem{Gevrey norm must be defined and adopted consistently in the text} 
fix the origin and are $C^\infty$-tangent to~$\Id$ at the origin.
We refer to Appendix~\ref{App:Gevrey} for the precise definition of
$\cU^{\al,L}$ and of a distance 
$\dist(\Phi,\Psi) = \norm{\Phi-\Psi}_{\al,L}$
which makes it a complete metric space.
We will prove the following.

%%%%%%%%%%%%%%%%%%%%%%%%%%%%%%%%%

\begin{Main} \label{theo.main} 
Fix $\al>1$ and $L>0$. For each $\ga>0$, there exist % exists 
a non-resonant
vector $\om \in \R^3$, 
a point $z \in  X$,
and a diffeomorphism $\Psi \in \cU^{\al,L}$ 
such that $\norm{\Psi-\Id}_{\al,L}\le \ga$ 
and $T = \Psi \circ S_\om$ satisfies $T^n(z) \underset{n\to +\infty}{ \longrightarrow} 0$. 
\end{Main}
%
%\margem{adapter l'\'enonc\'e au choix des d\'efinitions Gevrey} 

%%%%%%%%%%%%%%%%%%%%%%%%%%%%%%%%%
We do not know how to produce real analytic examples. Recall that 
% {\blue 
not even one example of a
% }
% not any example of 
real analytic symplectomorphism with a Lyapunov unstable non-resonant elliptic 
fixed point is known. 

%%%%%%%%
%%%%%%%%
% Je trouve qu'il faut ajouter ici qqch sur la littérature Gevrey
% dynamique... (David 13/03/2018)
%
For other instances of the use of Gevrey regularity with symplectic or Hamiltonian
dynamical systems, see \eg
\cite{Popov04},
\cite{hms}, 
\cite{MSetds},
\cite{MitevPopov},
\cite{mcwdGev},
\cite{BF}.
%
%%%%%%%%
%%%%%%%%

% {\red
Our construction easily extends to the case where $X=(\R^2)^n$ with $n\geq 3$, however we do not know how to adapt the method to the case $n=2$. 
As for the case $n=1$, there may well be no regular examples at all. Indeed if the rotation frequency at the fixed point is Diophantine, then a theorem by Herman (see \cite{FKherman}) implies that the fixed point is surrounded by invariant quasi-periodic circles, and thus is Lyapunov stable. The same conclusion holds by Moser's KAM theorem if the Birkhoff normal form at the origin is not degenerate \cite{moser}.  In the remaining case of a degenerate Birkhoff normal form with a Liouville frequency, there is evidence from \cite{AFLXZ} that the diffeomorphism should then be rigid in the neighborhood of the origin, that is, there exists a sequence of integers along which its iterates  converge to Identity near the origin, which clearly precludes the convergence to the origin of an orbit.
% }

Similar problems can be addressed where one searches for Hamiltonian
diffeomorphisms (or vector fields) with orbits whose $\al$-limit or
$\om$-limit have large Hausdorff dimension (or positive Lebesgue
measure) and in particular contain families of non-resonant invariant
Lagrangian tori instead of a single non-resonant fixed point.  
A specific example for Hamiltonian flows on $(\T\times \R)^3$ is
displayed in \cite{KS}, while a more generic one has been announced in
\cite{KG}.
In these examples, the setting is perturbative and the Hamiltonian
flow is non-degenerate in the neighborhood of the tori.  The methods
involved there are strongly related to Arnold diffusion and are  %, and so
completely different from ours.  
%
%\margem{ajouter qu'on peut facilement adapter les m\'ethodes du papier au cas des tores ?}

\section{Preliminaries and outline of the strategy}

%In polar coordinates we write $R_\om(r,\th)=(r,\th+\om)$. 
%We use the notations $r(s)=r$ and $\th(s) = \th [1]$ for $ s=(r,\th)$.

From now on we fix $\al>1$ and $L>0$. 
We also pick an auxiliary $L_1>L$.
For $z \in \R^2$ and $\nu>0$, we denote by $B(z,\nu)$ the 
closed 
ball relative to $\norm{\,.\,}_\infty$ centred at~$z$ with
radius~$\nu$.
Since $\al>1$, we have 

\begin{lemma}  \label{lemma.f}  
%
% Let $\al>1$ and $L>0$. 
%
There is a real $c=c(\al,L_1)>0$ such that,
for any $z \in \R^2$ and $\nu>0$, there exists a function $f_{z,\nu} \in G^{\al,L_1}(\R^2)$ which satisfies 
\begin{itemize}
\item $0\le f_{z,\nu}\le 1$, % $f_{z,\nu}(\R^2)\subset  [0,1]$ 
\item $f_{z,\nu}\equiv 1$ on $B(z,\nu/2)$,
\item $f_{z,\nu}\equiv 0$ on $B(z,\nu)^c$,
\item $\norm{f_{z,\nu}}_{\al,L_1} \leq \exp(c\,{\nu^{-\frac{1}{\al-1}}})$.
\end{itemize}
\end{lemma}

\begin{proof}
Use Lemma~3.3 of \cite{MSetds}.
\end{proof}

% In the following we fix $\al>1$ and $L>0$ and abbreviate $f_{z,\nu}^{\al,L}$ in $f_{z,\nu}$.

%%%%%%%%%%%%%%%%%%%%%%%%%%%%%%%%%
%%%%%%%%%%%%%%%%%%%%%%%%%%%%%%%%%

We now fix an arbitrary real $R>0$ and pick an auxiliary function
$\eta_R\in G^{\al,L_1}(\R)$ which is identically~$1$ on the interval
$[-2R,2R]$, identically~$0$ outside $[-3R,3R]$, and everywhere non-negative.
We then define $g_R \col \R^2 \to \R$ by the formula
%
% \margem{Peut-être vaut-il mieux utiliser~$g$ au lieu de~$g_R$ pour
%   éviter d'attirer trop l'attention sur~$R$}
%
\begla
g_R(x,y) \defeq x y \,\eta_R(x)\, \eta_R(y).
\edla
%
% Recall that our phase space is $X = (\R^2)^3$ viewed as a symplectic
% manifold by endowing each factor with the canonical symplectic
% $2$-form $\dd x \wedge \dd y$.

The following diffeomorphisms will be of constant use in this paper:

\begin{defin}   \label{definPhiznu}
For $(i,j)\in \{1,2,3\}$, $z \in \R^2$ and $\nu>0$, we
  denote by~$\Phi_{i,j,z,\nu}$ % $\in\cU^{\al,L}$ 
  the time-one map of the Hamiltonian flow generated by the function
% generated by the Hamiltonian function
%
%$$
%\tst s=(s_1,s_2,s_3) \mapsto
%
$\exp(-c\,\nu^{-\frac{2}{\al-1}}) f_{z,\nu} \otimes_{i,j} g_R$, %(s),
% $$
%
where 
$f_{z,\nu} \otimes_{i,j} g_R \col X \to \R$ stands for the function 
$$
s=(s_1,s_2,s_3) \mapsto f_{z,\nu} (s_i) g_R(s_j).
$$
\end{defin}

In the above definition, our convention for the Hamiltonian vector
field generated by a function~$H$ is
$X_H = \sum (-\frac{\pa H}{\pa y_i}\frac{\pa\;\;}{\pa x_i}
+ \frac{\pa H}{\pa x_i}\frac{\pa\;\;}{\pa y_i})$.
Note that the Hamiltonian $\exp(-c\,\nu^{-\frac{2}{\al-1}}) f_{z,\nu} \otimes_{i,j} g_R$ has compact support,
hence it generates a complete vector field and
Definition~\ref{definPhiznu} makes sense.
Actually, any $H\in G^{\al,L_1}(X)$ has bounded partial
derivatives, hence~$X_H$ is always complete;
the flow of~$X_H$ is made of Gevrey maps for which
estimates are given in Appendix~\ref{secappGevFlEsti}. In the case of
$\Phi_{i,j,z,\nu}$, % {\blue 
for~$\nu$ small enough %} 
we have
\beglab{ineqPhiijznu}
\Phi_{i,j,z,\nu}\in\cU^{\al,L} 
\quad\text{and}\quad
%\ens\text{and}\ens
%
\norm{\Phi_{i,j,z,\nu}-\Id}_{\al,L} \le K \exp(-c\,\nu^{-\frac{1}{\al-1}}),
\edla
with $K \defeq C \norm{g_R}_{\al,L_1}$, where~$C$ is independent from
$i,j,z,\nu$ and stems from~\eqref{ineqPhiuId}.
Here are the properties which make the~$\Phi_{i,j,z,\nu}$'s precious. To alleviate the notations, we  state them 
for $\Phi_{2,1,z,\nu}$ but similar properties hold for each diffeomorphism~$\Phi_{i,j,z,\nu}$.

\begin{lemma} \label{lemma.phi} 
Let $z \in \R^2$ and $\nu>0$. Then~$\Phi_{2,1,z,\nu}$ satisfies:
\medskip

\begin{itemize} 
\item[(a)] For every $(s_1,s_2,s_3) \in X$ such that $s_2 \in B(z,\nu)^c$,  
$$\Phi_{2,1,z,\nu}(s_1,s_2,s_3)=(s_1,s_2,s_3).$$
\\[-3ex]

\item[(b)] For every $x_1 \in \R$, $s_2\in\R^2$ and $s_3\in\R^2$,
 $$\Phi_{2,1,z,\nu}((x_1,0),s_2,s_3)=((\tx_1,0),s_2,s_3)
\ens\text{with} \ens 
\abs{\tx_1} \le \abs{x_1}.$$
\smallskip

\item[(c)]  For every $x_1 \in [-2R,2R]$, $s_2 \in B(z,\nu/2)$ and
  $s_3 \in \R^2$,
$$\Phi_{2,1,z,\nu}((x_1,0),s_2,s_3)=((\tx_1,0),s_2,s_3)
\ens\text{with} \ens 
\abs{\tx_1} \le \ka\abs{x_1},$$ 
where $\ka \defeq 1- \tdemi\exp(-c\,\nu^{-\frac{2}{\al-1}})$.
\end{itemize} 
\smallskip

\end{lemma}

\begin{proof}
%
%Formula~(2.11) of \cite{hms} or formula~(13) of \cite{MSetds} yield an
The dynamics of the flow generated by $f_{z,\nu}\otimes_{2,1} g_R$ can easily be understood from those of  the flows generated by~$f_{z,\nu}$ alone on the second factor~$\R^2$ and
by~$g_R$ alone on the first factor~$\R^2$. Indeed 
$$\Phi_{f \otimes_{2,1} g}(z_1,z_2)=(\Phi^{f(z_2)g}(z_1),\Phi^{g(z_1)f}(z_2))$$
where $\Phi^h$ denotes the time one map associated to the Hamiltonian $h$. 

The properties (a)-(b)-(c) immediately follow from the latter expression. 
\end{proof}

%Hence, if $z_2 \in B(z,\nu/2)$, we have that the action on the second coordinate is Identity, while in the first coordinate, $(0,0)$ is a hyperbolic
%fixed point and the axis $x_1=0$ (which contains the stable manifold)
%is invariant. 
%
%More precisely, one can check that, in the first factor~$\R^2$, an
%arbitrary point $(x_1,0)$ is sent by the the time-$t$ map of~$g_R$ to
%$(x_1(t),0)$ with $\frac{\dd x_1}{\dd t}$ of the same sign as~$x_1$;
%
%moreover, if $x_1\in[-2R,2R]$ and $t\ge0$, then
%
%$x_1(t) = \ee^{-t} x_1$.

%The conclusion follows from the explicit formula to which we alluded above.
%
%%%%%%%%%%%%%%%%%%%%%%%%%%%%%%%%%

From now on, we denote simply by $\abs{\,.\,}$ the
$\norm{\,.\,}_\infty$ norm in~$\R^2$ or in $X=\R^6$, and by $B(s,\rho)$
the corresponding closed ball centred at~$s$ with radius~$\rho$ (the
context will tell whether it is in~$\R^2$ or~$\R^6$).

Here is a brief outline of the strategy for proving
Theorem~\ref{theo.main} and obtaining, inductively, the required~$\Psi$, $z$ and~$\om$:
\begin{itemize}
\item
The diffeomorphism~$\Psi$ in   Theorem~\ref{theo.main} will be obtained
as an infinite product (for composition) of diffeomorphisms of the form~$\Phi_{i,j,z,\nu}$,
with smaller and smaller values of~$\nu$ so as to derive convergence in~$\cU^{\al,L}$ from~\eqref{ineqPhiijznu}.
\item
  On the other hand, $R$ will be kept fixed and the initial
  condition~$z$ will be obtained as the limit of a sequence contained
  in the ball $B(0,R) \subset X$.
\item
As for the non-resonant frequency vector~$\om$ in
Theorem~\ref{theo.main}, it will be obtained as a limit of vectors
with rational coordinates with larger and larger denominators, so as
to make possible a kind of ``orbit synchronization'' at each step of the construction.
\end{itemize}

%%%%%%%%%%%%%%%%%%%%%%%%%%%%%%%%%
%%%%%%%%%%%%%%%%%%%%%%%%%%%%%%%%%
%%%%%%%%%%%%%%%%%%%%%%%%%%%%%%%%%

\section{The attraction mechanism} % {The ``attraction mechanism''}

Starting from a point $z=((x_1,0),(x_2,0),(x_3,0))$, the mechanism of attraction of the point to the origin 
is an alternation between bringing closer to zero the $x_1$,$x_2$ or $x_3$ coordinates when all the coordinates of the point  
come back to the horizontal axes. The main ingredient is the following lemma,
where 
we use shortcut notation $\Phi_{i,j,x,\nu}$ for $\Phi_{i,j,(x,0),\nu}$
and, for two integers $Q_1,Q_2$, the notation
$Q_1\vert Q_2$ stands for ``$Q_1$ divides $Q_2$''.

\begin{lemma} \label{sublemma_ascenseur}
Let $\omega=(P_1/Q_1,P_2/Q_2,P_3/Q_3)\in\Q^3$ with $P_i, Q_i$ coprime
positive integers and
$$
 z=((x_1,0),(x_2,0),(x_3,0)) \in B(0,R).
$$
Set
 \begin{align*}
T_1&=\Phi_{2,1,x_2,Q_2^{-3}} \circ  \Phi_{1,3,x_1,Q_1^{-3}} \circ S_{{\omega}} \\
T_2&=\Phi_{3,2,x_3,Q_3^{-3}} \circ  \Phi_{2,1,x_2,Q_2^{-3}} \circ S_{{\omega}} \\
T_3&=\Phi_{1,3,x_1,Q_1^{-3}} \circ  \Phi_{3,2,x_3,Q_3^{-3}} \circ S_{{\omega}} 
\end{align*}

Then the following properties hold.

\medskip

\noindent {\bf I)} If $Q_3\vert Q_1$ and $Q_1\vert Q_2$, and  
$$
x_1\geq 1/Q_1,\quad  x_2\geq0,\quad  x_3\geq1/Q_3,
$$
then there exists $N$ such that 
$$
T_1^{N}(z)=((\hat{x}_1,0),(\hat{x}_2,0),(\hat{x}_3,0))
$$
with 
$$
0\leq\hat{x}_1\leq x_1/2,\quad 0\leq\hat{x}_2= x_2,\quad 0\leq\hat{x}_3\leq x_3,
$$ and 
$\abs{T_1^{m}(z)_i} \leq x_i$ for all $m \in \{0,\ldots,N\}$.

%If  $x_2\geq1/Q_2$, $x_1\geq1/Q_1$ and $Q_2=K Q_1$, there exists $N$ such that $T_1^{N}(z)=\hat{z}$ with $\hat{z}=((\hat{x}_1,0),(\hat{x}_2,0),(\hat{x}_3,0))$ such that $\hat{x}_1\leq x_1/2$,  $\hat{x}_2= x_2$,  $\hat{x}_3\leq x_3$ and $r(T_1^{m}(z)_i) \leq x_i$ for all $m \in [0,N]$. 

\medskip

\noindent {\bf II)} If $Q_1\vert Q_2$ and $Q_2\vert Q_3$, and  
$$
x_1\geq 0,\quad  x_2\geq1/Q_2,\quad  x_3\geq1/Q_3,
$$
then there exists $N$ such that 
$$T_2^{N}(z)=((\hat{x}_1,0),(\hat{x}_2,0),(\hat{x}_3,0))$$
with 
$$
0\leq\hat{x}_1\leq x_1,\quad 0\leq\hat{x}_2\leq x_2/2,\quad 0\leq\hat{x}_3= x_3,
$$ and 
$\abs{T_2^{m}(z)_i} \leq x_i$ for all $m \in \{0,\ldots,N\}$. 

\medskip 

\noindent {\bf III)}  If $Q_2\vert Q_3$ and $Q_3\vert Q_1$, and  
$$
x_1\geq 1/Q_1,\quad  x_2\geq0,\quad  x_3\geq1/Q_3,
$$
then there exists $N$ such that 
$$T_3^{N}(z)=((\hat{x}_1,0),(\hat{x}_2,0),(\hat{x}_3,0))$$
with 
$$
0\leq\hat{x}_1= x_1,\quad 0\leq\hat{x}_2\leq x_2,\quad 0\leq\hat{x}_3\leq x_3/2,
$$ and 
$\abs{T_3^{m}(z)_i} \leq x_i$ for all $m \in \{0,\ldots,N\}$.

%If  $x_1\geq1/Q_1$, $x_3\geq1/Q_3$ and $Q_1=K Q_3$, there exists $N$ such that $T_3^{N}(z)=\hat{z}$ with $\hat{z}=((\hat{x}_1,0),(\hat{x}_2,0),(\hat{x}_3,0))$ such that $\hat{x}_1= x_1$,  $\hat{x}_2\leq x_2$,  $\hat{x}_3\leq x_3/2$ and $r(T_3^{m}(z)_i) \leq x_i$ for all $m \in [0,N]$. 

\end{lemma}

\begin{proof} We will prove the Lemma  for $T_2$ since it will be the first map that we will use in the sequel. The proof for the maps $T_1$ and $T_3$ follows exactly the same lines. 

\vskip2mm

%The hypothesis $x_3\geq 1/Q_3$ implies that as the orbit of $z_3=(x_3,0)$ is rotating under the action of $R_{\om_3}$ it only enters the $Q_3^{-3}$ neighborhood of $z_3$ at times that are multiples of $Q_3$ where it lands back on $z_3$. Lemma \ref{lemma.phi} thus implies that 
%$$T_2^{m}(z)=(z_{1,m},z_{2,m},z_3)$$
%with $z_{2,lQ_3} =(x_{2,l},0)$ with $0<x_{2,l+1}  \leq (1-e^{-Q_3^6}/2) x_{2,l}$ (the latter comes from  (c) of Lemma \ref{lemma.phi}).
%For all other iterates of $T_2$, $z_{2,l}$ is just rotating (for multiples of $Q_2$ where $\Phi_{2,1,x_2,Q_2^{-3}}$ may be active it follows from Lemma \ref{lemma.phi} (b) that due to our choice of $z$, $z_2$ is unchanged by $\Phi_{2,1,x_2,Q_2^{-3}}$). It also follows from Lemma \ref{lemma.phi} (b) that $z_{1,l}$ will be rotating except for multiples of $Q_2$ when it comes back to the $x$-axis and gets closer to the origin before it stops moving as $z_{2,l}$ exits the ball $B(z_2,1/Q_2^3)$. 
% 
%We let $L$  be the smallest integer such that  $|x_{2,LQ_3}|\leq|x_3|/2$ and get the conclusion of Lemma  \ref{sublemma_ascenseur} with $N=LQ_3$. 

The hypothesis $x_2\geq 1/Q_2$  implies that the orbit of $z_2=(x_2,0)$ under the rotation  $R_{\om_2}$  enters the $Q_2^{-3}$ neighborhood of $z_2$  only at times that are multiples of $Q_2$. Moreover $R_{\om_2}^{\ell Q_2}(z_2)=z_2$. A similar remark
holds for $z_3$.

\vskip2mm

Since $Q_3>Q_2$, we consider the action of $\cT\defeq\Phi_{2,1,x_2,Q_2^{-3}}\circ S_{\om}$ first. 
Since $Q_1\mid Q_2$, if $s=(s_1,s_2,s_3)$ with $s_1=(u_1,0)$ and $s_2=(u_2,0)$, by Lemma \ref{lemma.phi}:
$$
\cT^{m}(s)=(s_{1,m},R^m_{\om_2}(s_2),R^m_{\om_3}(s_3)) \quad 
\text{for all $m\in\N$,}
%
% \forall m\in\N,
$$
with
$$
\abs{s_{1,m}}\leq \abs{s_1}.
$$

\vskip1mm

Consider now the orbit of $z$ under the full diffeomorphism $T_2$. Since $Q_2\mid Q_3$,  the
previous remark shows that one has to take the effect of $\Phi_{3,2,x_3,Q_3^{-3}}$ into account only for
the iterates of order $m=\ell Q_3$. One therefore gets
$$
T_2^{m}(z)=(z_{1,m},z_{2,m},R^m_{\om_3}(z_3)),\quad 
\text{for all $m\in\N$,}
%
% \forall m\in\N,
$$
where in particular
$z_{2,\ell Q_3} =(x_{2,\ell Q_3},0)$ 
with 
$$
0<x_{2,(\ell+1) Q_3}  \leq (1-\tdemi \exp(- c Q_3^{\frac{6}{\al-1}}) )x_{2,\ell Q_3},
$$
and where 
$$
z_{2,\ell Q_3+\ell'} = R_{\om_2}^{\ell'}(z_{2,\ell Q_3}),\quad 1\leq \ell'\leq Q_3-1,
$$ 
$$
\abs{z_{1,m}}\leq x_1, \quad 
\text{for all $m\in\N$.}
%
% \forall m\in\N.
$$

\vskip2mm
 
We let $L$  be the smallest integer such that  $0<x_{2,LQ_3}\leq x_2/2$ and get the conclusion with $N=LQ_3$. 
\end{proof}

%%%%%%%%%%%%%%%%%%%%%%%%%%%%%%%%%
%%%%%%%%%%%%%%%%%%%%%%%%%%%%%%%%%
%%%%%%%%%%%%%%%%%%%%%%%%%%%%%%%%%

\section{Proof of Theorem \ref{theo.main}}\label{Sec:proof}

The proof is based on an iterative process (Proposition~\ref{prop:iterative}) which is itself based
on the following preliminary result. 
For positive integers $q_1,q_2,q_3$, the notation $q_3\vert q_1\vert
q_2$ means ``$q_3$ divides~$q_1$ and $q_1$ divides~$q_2$''.

\begin{prop}  \label{prop:inductive.step} 

Let $\om=(p_1/q_1,p_2/q_2,p_3/q_3) \in \Q_+^3$ 
with $q_3\vert q_1\vert q_2$ 
and
%
% \[
%
$z=((x_1,0),(x_2,0),(x_3,0)) \in B(0,R)$
%
%\ens\text{with $x_1, x_2, x_3 >0$.} % $x_i>0$ for $i\in\{1,2,3\}$.}
%
%\]
%
with $x_1, x_2, x_3 >0$ and $x_2\ge 1/q_2$. 
Then, for any $\eta>0$, there exist 

\begin{itemize}

\item[$(a)$]   $\ov{\om} =(\bp_1/\bq_1,\bp_2/\bq_2,\bp_3/\bq_3)$
%
% \margem{$\abs{\bom-\om}\leq \eta$ ou $\abs{\bom-\om}\leq \eta^3$?}
%
such that $\bq_3\vert\bq_1\vert\bq_2$,
the orbits of the translation of vector $\ov \om$ on $\T^3$ are
$\eta$-dense % and $\ov q_2$-periodic;
and $\abs{\ov{\om}-\om}\leq \eta$;
%  $\bom_2=P_2/Q_2$ and $\bom_3=P_3/Q_3$ with $Q_3=\bar{K}Q_2$ for some $\bar{K} \in \N$.  

\item[$(b)$]  $\bz=((\bx_1,0),(\bx_2,0),(\bx_3,0))$ such that $0<\ov x_i\leq x_i /2$ for every $i \in \{1,2,3\}$
%and $\bx_2\geq1/\bq_2$.%and $0\leq\leq\bx_2$ %$|\bx_3|>1/Q_3$, and .
%
and $\bx_2 \ge 1/\bq_2$;

\item[$(c)$] $z' \in X$, 
%
% $\ha q_2\geq q_2$, 
%
% \margem{A-t-on besoin de tant de précision sur~$\ha x_1$?}
%
$\ha x_1\in(\ov x_1+\frac{1}{\ov q_1^3},x_1)$
and $N \in \N$,
such that $\abs{ z'-z}\leq \eta$
and the diffeomorphism 
$$
\cT=  \Phi_{2,1,\bx_2,\bq_2^{-3}} \circ  \Phi_{1,3,\ha x_1,\bq_1^{-3}} 
\circ  \Phi_{3,2,x_3,\bq_3^{-3}} 
\circ \Phi_{2,1,x_2,q_2^{-3}} % \Phi_{2,1,x_2,\ha q_2^{-3}} 
\circ S_{\bom}
$$
satisfies  
$$
{\cT}^{N}(z')=\bz
$$  and $\abs{{\cT}^{m}(z')_i} \leq (1+\eta)x_i$ for $m \in \{0,\ldots,N\}$. 
%\item[$(d)$]  $| \Phi_{2,1,\bx_2,\bq_2^{-3}} \circ  \Phi_{1,3,x_1,\bq_1^{-3}} \circ  \Phi_{3,2,x_3,\bq_3^{-3}} \circ  \Phi_{2,1,x_2,q_2^{-3}}-  \Phi_{2,1,x_2,q_2^{-3}} |\leq \eta$.
\\[-1.5ex]

\end{itemize} 
%
%Note that for $s$ such that $r(s_i)\leq \bx_i$ for every $i=1,2,3$, then $\bT(s)=\Phi_{2,1,\bx_2,\bq_2^{-3}} \circ S_{\bom}(s)$.
%
Moreover, $\bq_1$, $\bq_2$ and~$\bq_3$ % and~$\ha q_2$ 
can be taken arbitrarily large.
\end{prop}

\begin{proof}[Proof of Proposition \ref{prop:inductive.step}.] 
We divide the proof into three steps.

\paraga 
%
% First choose a rotation vector $\ha{\omega}$ of the form 
% $$
% \ha{\omega}=(p_1/q_1,\ha p_2/\ha q_2 ,\ha p_3/\ha q_3)
% $$ 
% with coprime $\ha p_i$ and $\ha q_i$, and
% \begin{equation}\label{eq:choice}
% q_1\vert \ha q_2,\qquad \ha q_2\vert \ha q_3,
% \qquad x_2\geq \frac{1}{\ha q_2},\qquad x_3\geq \frac{1}{\ha q_3}, \qquad \frac{1}{\ha q_3}<\eta,
% \end{equation}
% and
% $\abs{\ha \om-\om}<\eta$.
% Set 
% $$
% \ha T_2=\Phi_{3,2,x_3,\ha q_3^{-3}} \circ  \Phi_{2,1,x_2,\ha q_2^{-3}} \circ S_{\ha{\omega}}.
% $$
%
First choose coprime integers $\ha p_3$ and~$\ha q_3$ with~$\ha q_3$
large multiple of~$q_2$, so that
\begin{equation}\label{eq:choice}
q_1\vert q_2\vert \ha q_3,
\qquad x_2\geq \frac{1}{q_2},
\qquad x_3\geq \frac{1}{\ha q_3}, 
\qquad \frac{1}{\ha q_3}<\eta
\end{equation}
and the new rotation vector
$$
\ha{\omega}=(p_1/q_1,p_2/q_2 ,\ha p_3/\ha q_3)
$$ 
satisfies
$\abs{\ha \om-\om}<\eta$.
Set 
$$
\ha T_2=\Phi_{3,2,x_3,\ha q_3^{-3}} \circ  \Phi_{2,1,x_2,q_2^{-3}} \circ S_{\ha{\omega}}.
$$
By Lemma  \ref{sublemma_ascenseur} {\bf II)}, there exist $\ha{N}\in\N$ 
and $\ha{z}=((\ha{x}_1,0),(\ha{x}_2,0),(\ha{x}_3,0))$  such that 
$\ha T_2^{\ha{N}}(z)=\ha{z}$, 
with 
$$
\ha{x}_1\leq x_1,\quad\ha{x}_2\leq x_2/2,\quad\ha{x}_3= x_3,
$$
and 
$\abs{\ha T_2^{m}(z)_i}  \leq x_i$ for all $m \in \{0, \ldots,\ha N\}$. 

\paraga Next, % fix a rotation vector of the form
% $$
% \til{\omega}=(\til p_1/\til q_1,\ha p_2/\ha q_2 ,\ha p_3/\ha q_3)
% $$ 
consider a vector of the form
$$
\til{\omega}=(\til p_1/\til q_1,p_2/q_2 ,\ha p_3/\ha q_3)
$$ 
with coprime $\til p_1$ and $\til q_1$, and
\begin{equation}\label{eq:K}
\ha q_3\vert \til q_1,\qquad  \ha x_1>\frac{1}{\til q_1},\qquad  \frac{\ha q_3}{\til q_1}<\eta,
\end{equation}
so that in particular
\begin{equation}\label{eq:ineg}
\frac{\ha x_1}{2}>\frac{1}{\til q_1^3}.
\end{equation}
Set
$$
\til{T_3}=\Phi_{1,3,\ha x_1,\til q_1^{-3}} \circ  \Phi_{3,2,x_3,\ha q_3^{-3}} \circ S_{\til{\omega}}.
$$ 
By Lemma  \ref{sublemma_ascenseur} {\bf III)}, there exist $\til{N}\in\N$ and
$\til {z}=((\til {x}_1,0),(\til {x}_2,0),(\til {x}_3,0))$ such that $\til {T_3}^{\til {N}}(\ha{z})=\til {z}$  with
\be
\til {x}_1= \ha{x}_1,\quad \til {x}_2\leq \ha{x}_2\leq x_2/2,\quad \til {x}_3\leq \ha{x}_3/2=x_3/2,
\eeq
and 
$\abs{{\til  T_3}^{m}(\ha{z})_i}  \leq \ha{x}_i$ for all $m \in \{0,\ldots,\til  N\}$. 

Define now
$$
\bbT= \Phi_{1,3,\ha x_1,\til q_1^{-3}} 
\circ \Phi_{3,2,x_3,\ha q_3^{-3}} 
\circ \Phi_{2,1,x_2,q_2^{-3}} % \Phi_{2,1,x_2,\ha q_2^{-3}} 
\circ  S_{\til{\omega}}.
$$ 
Choosing $\til q_1$ in (\ref{eq:K}) large enough and $\til p_1$ properly, 
one can assume that  $\til \om$ is arbitrarily close to $\ha \om$,  so that $S_{\til\om}$ is arbitrarily 
$C^0$-close to $S_{\ha\om}$ on the ball $B=B(0,\abs{z}+1)$, and moreover that
$\Phi_{1,3,\ha x_1,\til q_1^{-3}}$ is arbitrarily 
$C^0$-close to $\Id$ on $B$. As a consequence, one
can assume that $\bbT$ is arbitrarily $C^0$-close to $\ha T_2$ on  $B$. 
Hence one can choose $\til\om$ with $\abs{\til \om-\om}<\eta$ such that there exists $\bbz$ with 
$\abs{\bbz- z}<\eta$ which satisfies
$$
\bbT^{\ha{N}}(\bbz)=\ha{z},\qquad 
\abs {{\bbT}^{m}(\bbz)_i}  \leq (1+\eta) x_i
\quad\text{for all $m \in \{0,\ldots,\ha N\}$.}
%, \quad \forall m \in \{0,\ldots,\ha N\}.
$$
Moreover, using Lemma~\ref{lemma.phi},  one proves by induction that: 
$$
{\bbT}^{m}(\ha z)_2\in B(x_2,\ha q_2^{-3})^c,\qquad
\bbT^{m}(\ha{z})=\til{T_3}^{m}(\ha{z})
\quad\text{for all $m \in \{0,\ldots,\til N\}$.}
%,\quad \forall m \in \{0,\ldots,\til N\}.
$$
As a consequence   
$$
{\bbT}^{\ha{N}+\til{N}}(\bbz)=\til T_3^{\til N}(\ha z)=\til{z}
$$ 
and $\abs{{\bbT}^{m}(\bbz)_i}< (1+\eta) x_i$ for all $m\in \{0,\ldots, \ha{N}+\til{N}\}$.

\paraga It remains now to perturb $\bbT$ in the same way as above to bring
the first component of $\til z$ closer to the origin.
Consider coprime integers~$\ov p_2$ and~$\ov q_2$ such that % $\pgcd(\ov p_2,\ov q_2)=1$ and
\begin{equation}\label{eq:ovK}
\til q_1\vert \ov q_2,
\qquad x_2\geq 1/\ov q_2,
\qquad \til x_2\geq 1/\ov q_2,
\qquad \frac{\til q_1}{\ov q_2}<\eta,
\end{equation}
and such that the vector
\begin{equation}\label{eq:omega}
\ov{\omega}=(\til p_1/\til q_1,\ov p_2/\ov q_2 ,\ha p_3/\ha q_3)
\end{equation}
satisfies $\abs{\ov \om-\om}<\eta$.
Set now
$$
\cT=  \Phi_{2,1,\tx_2,\bq_2^{-3}} 
\circ \Phi_{1,3,\ha x_1,\til q_1^{-3}} 
\circ \Phi_{3,2,x_3,\ha q_3^{-3}} 
\circ \Phi_{2,1,x_2,q_2^{-3}} % \Phi_{2,1,x_2,\ha q_2^{-3}} 
\circ S_{\ov\om}.
$$

%Finally, let $\bar{x}_2=\tilde{x}_2$ and take $\bom=(\bp_1/\bq_1,\bp_2/\bq_2,\bp_3/\bq_3)$ 
%where $\bq_2=\bar{K} \bq_1$ for $\bar{K} \gg 1$ such that $1/\bq_2\leq\bar{x}_2$. 

As above, a proper choice of $\ov p_2$ and $\ov q_2$ satisfying (\ref{eq:ovK}) makes $\cT$ arbitrarily $C^0$ close to $\bbT$ and
yields the existence of a $z'\in X$ such that $\abs{z'- z}< \eta$, satisfying 
$$
\cT^{\ha{N} +\til{N}}(z')=\til{z},\qquad \abs{{\cT}^{m}(z')_i} <
(1+\eta) x_i
\quad\text{for all $m\in \{0,\ldots, \ha{N}+\til{N}\}$.}
%,\quad \forall m\in \{0,\ldots, \ha{N}+\til{N}\}.
$$ 
Set 
$$
\bT_1=\Phi_{2,1,\tx_2,\bq_2^{-3}} \circ  \Phi_{1,3,\ha x_1,\bq_1^{-3}}  \circ S_{\ov \om}.
$$
Using Lemma  \ref{lemma.phi} and Lemma~\ref{sublemma_ascenseur} {\bf I)}, one proves by induction that now for $m\geq 0$:
$$
\cT^m(\til z)_2\in B(x_2, \ov q_2^{-3})^c,\qquad \cT^m(\til z)_3\in B(x_3, \ov q_3^{-3})^c,\qquad 
\cT^m(\til z)=\ov T_1^m(\til z).
$$
By Lemma~\ref{sublemma_ascenseur} {\bf I)} there exists 
 $\ov N$ such that 
 $$
 \bT_1^{\ov N}(\til{z})=\bz=((\ov{x}_1,0),(\ov{x}_2,0),(\ov{x}_3,0))
 $$
 with
  $$
  \ov{x}_1\leq \til{x}_1/2\leq x_1/2,\quad \ov{x}_2= \til{x}_2\leq x_2/2,\quad\ov{x}_3\leq \til{x}_3\leq x_3/2,
  $$ 
  and $\abs{({\ov T_1}^{m}(\til{z})_i} \leq \tx_i \leq x_i$ for all $m \in \{0,\ldots,\ov N\}$. 
As a consequence, setting $N=\ha N+\til N+\ov N$:
 $$
 {\cT}^{N}(z')=\bz,\qquad
\abs{{\cT}^{m}(z')_i} \leq (1+\eta)x_i
\quad\text{for all $m \in \{0,\ldots,N\}$.}
%,\quad  \forall m \in \{0,\ldots,N\}.
$$

We finally change the notation of (\ref{eq:omega}) and write
$$
{\ov \om}=(\ov\om_1,\ov\om_2,\ov\om_3)=(\ov p_1/\ov q_1,\ov p_2/\ov q_2 ,\ov p_3/\ov q_3),
$$
so that in particular $\til q_1=\ov q_1$, $\ha q_3=\ov q_3$ and
$$
\ov q_3\mid \ov q_1,\qquad \ov q_1\mid \ov q_2.
$$
Hence the orbits of $S_{\ov \om}$ are $\ov q_2$-periodic. Moreover,
from (\ref{eq:ineg}) and the equality $\ha x_1=\til x_1$, one deduces
$$
\ha x_1-\ov x_1>\frac{1}{\ov q_1^3}.
$$

Note finally that the last conditions in (\ref{eq:choice}), (\ref{eq:K}) and (\ref{eq:ovK}) now read
$$
\frac{1}{\ov q_3}<\eta,\qquad\frac{\ov q_3}{\ov q_1}<\eta,\qquad \frac{\ov q_1}{\ov q_2}<\eta.
$$
Fix $(\th_1,\th_2,\th_3)\in\T^3$ and recall that $\ov q_3\mid \ov q_1$ and $\ov q_1\mid \ov q_2$.
By the first inequality one can first find  $\ell_3\in\N$
such that $R^{\ell_3}_{\ov \om_3}(0)$ is $\eta$-close to $\th_3$. Then, by the second inequality  there is an $\ell_1\in\N$
such that $R_{\ov\om_1}^{\ell_1\ov q_3+\ell_3}(0)$ is $\eta$-close to $\th_1$. Finally, by the last inequality there
is an $\ell_2\in\N$ such that $R_{\ov\om_2}^{\ell_2\ov q_1+\ell_1\ov q_3+\ell_3}(0)$ is $\eta$-close to $\th_2$.
This proves that $S_{\ov \om}^{\ell_2\ov q_1+\ell_1\ov q_3+\ell_3}(0,0,0)$ is $\eta$-close to $(\th_1,\th_2,\th_3)$,
so that the orbits of $S_{\ov \om}$ are $\eta$-dense on $\T^3$. This concludes the proof.
 \end{proof}

% \margem{Fin de la partie r\'evis\'ee par JP} 

\begin{defin} 
%
%{\blue 
Given $z=(z_1,z_2,z_3) \in X$, we say that a diffeomorphism~$\Phi$ of~$X$ is $z$-{\it admissible}
if $\Phi \equiv \Id$ on
$$\{ s \in X : \abs{s_i}\leq  \tfrac{11}{10} \abs{z_i}, i=1,2,3 \} .$$
%}
%
% For $z=(z_1,z_2,z_3) \in X$ we use the notation
% %
% $$ \cB(z)\defeq\{ s \in X : \abs{s_i}\leq  \frac{11}{10} \abs{z_i}, i=1,2,3 \} $$
% %
% and we say that a diffeomorphism~$\Phi$ of~$X$ is $z$-{\it admissible}
% if $\Phi \equiv \Id$ on~$\cB(z)$.
%
\end{defin}

\begin{prop}   \label{prop:iterative} 
Let $\om=(p_1/q_1,p_2/q_2,p_3/q_3) \in \Q_+^3$ 
with $q_3 \vert q_1 \vert q_2$
and 
%
%\[
%
$ z=((x_1,0),(x_2,0),(x_3,0)) \in B(0,R) $
%
% \ens\text{with $x_1,x_2,x_3>0$}
%
%\]
%
with $x_1,x_2,x_3>0$ and $x_2\ge 1/q_2$. % \dfrac{1}{q_2}$.
Suppose $\Phi \in \cU^{\al,L}$ is $z$-admissible and
$\norm{\Phi_{2,1,x_2,q_2^{-3}} \circ \Phi-\Id}_{\al,L} 
%{\blue 
<\eps$,
%}$, % \le\eps$,
where~$\eps$ is defined by Lemma~\ref{lemComposFlow}, and let
\[
T \defeq \Phi_{2,1,x_2,q_2^{-3}} \circ \Phi \circ  S_{\om}.
\]
Assume that $z_0\in X$ and $M\ge1$ are such that $T^M(z_0)=z$.
Then, for any $\eta>0$, there exist 
\begin{itemize}

\item[$(a)$] $\bom =(\bp_1/\bq_1,\bp_2/\bq_2,\bp_3/\bq_3)$
% %
% %\margem{$\abs{\bom-\om}\leq \eta^3$ ou $\abs{\bom-\om}\leq \eta$?}
% %
%   such that $\abs{\bom-\om}\le \eta$ % \eta^3$
%   and 
% %
% % $\bom_2=P_2/Q_2$ and $\bom_3=P_3/Q_3$ with $Q_3=\bar{K}Q_2$ for some $\bar{K} \in \N$.
% %
%   the orbits of the translation of vector $\bom$ on $\T^3$ are
%   $\eta$-dense and $\bq_2$-periodic;
%
  such that $\bq_3\vert \bq_1 \vert \bq_2$,
  the orbits of the translation of vector $\bom$ on $\T^3$ are
  $\eta$-dense
  and $\abs{\bom-\om}\le \eta$;
%
% $\bom_2=P_2/Q_2$ and $\bom_3=P_3/Q_3$ with $Q_3=\bar{K}Q_2$ for some $\bar{K} \in \N$.
%
\\[-1.5ex]

\item[$(b)$]  $\bz=((\bx_1,0),(\bx_2,0),(\bx_3,0))$ such that
  $0 < \bx_i \le x_i/2$ for every $i \in \{1,2,3\}$ and
  $\bx_2 \ge 1/\bq_2$;  % >1/\bq_2$;  
% and $0<\leq\bx_2$ % $|\bx_3|>1/Q_3$, and .
%
\\[-1.5ex]

\item[$(c)$] $\bar{z}_0 \in X$ such that $\abs{ \bar{z}_0-z_0}\leq
  \eta$, 
and $\bM \ge M$, % \in \N$,
and $\bPhi \in \cU^{\al,L}$ $\bz$-admissible, 
so that the diffeomorphism 
\[
\bT \defeq \Phi_{2,1,\bx_2,\bq_2^{-3}} \circ  \bPhi \circ S_{\bom}
\]
satisfies  $\bT^{\bM}(\bz_0)=\bz$  and $\abs{{\bT}^{m}(\bz_0)_i} \leq (1+\eta)x_i$ for all $m \in \big\{M,\ldots,\bM\big\}$. 
\\[-1.5ex]

\item[$(d)$] Moreover, 
%
% $\norm{\bPhi - \Phi}_{\al,L} \le \eta$.
%
$\norm{\Phi_{2,1,\bx_2,\bq_2^{-3}} \circ  \bPhi - \Phi_{2,1,x_2,q_2^{-3}} \circ  \Phi}_{\al,L} \le \eta$.
 \end{itemize} 

\end{prop} 

\begin{proof}[Proof of Proposition \ref{prop:iterative}.]
  Take $\bom, \bz, N, z', %\ha q_2, 
\ha x_1$ as in Proposition \ref{prop:inductive.step} and
  let
$$
\cT=  \Phi_{2,1,\bx_2,\bq_2^{-3}} 
\circ  \Phi_{1,3,\ha x_1,\bq_1^{-3}} 
\circ  \Phi_{3,2,x_3,\bq_3^{-3}} 
\circ  \Phi_{2,1,x_2,q_2^{-3}} % \Phi_{2,1,x_2,\ha q_2^{-3}} 
\circ S_{\bom}
$$
so that ${\cT}^{N}(z')=\bz$  and $\abs{{\cT}^{m}(z')_i} \leq (1+\eta)x_i$ for all $m \in \{0,\ldots,N\}$. 
If we define 
$$
\bT= \Phi_{2,1,\bx_2,\bq_2^{-3}} 
\circ \Phi_{1,3,\ha x_1,\bq_1^{-3}} 
\circ \Phi_{3,2,x_3,\bq_3^{-3}} 
\circ \Phi_{2,1,x_2,q_2^{-3}} % \Phi_{2,1,x_2,\ha q_2^{-3}} 
\circ \Phi \circ S_{\bom}
$$
then, since $\Phi$ is $z$-admissible and $\abs{z-z'}<\eta$, we get
${\cT}^{m}(z')={\bT}^{m}(z')$ for all $m \in \{0,\ldots,N\}$, hence
${\bT}^{N}(z')=\bz$ and $\abs{{\bT}^{m}(z')_i} \leq (1+\eta)x_i$
for all $m \in \{0,\ldots,N\}$.

\smallskip

Let
\beglab{eqPhipdtinf}
\bPhi\defeq 
\Phi_{1,3,\ha x_1,\bq_1^{-3}} 
\circ \Phi_{3,2,x_3,\bq_3^{-3}} 
\circ \Phi_{2,1,x_2,q_2^{-3}} % \Phi_{2,1,x_2,\ha q_2^{-3}} 
\circ \Phi,
\edla
so that, indeed, $\bT = \Phi_{2,1,\bx_2,\bq_2^{-3}} \circ \bPhi \circ S_\bom$.
Notice that we can write
$\Phi_{2,1,\bx_2,\bq_2^{-3}} \circ\bPhi = \Phi^{u_3}\circ \Phi^{u_2} \circ \Phi^{u_1} \circ \Psi$
(notation of Lemma~\ref{lemComposFlow}), where
$\Psi = \Phi_{2,1,x_2,q_2^{-3}} \circ \Phi$ and the Gevrey-$(\al,L_1)$
norms of $u_1$, $u_2$, $u_3$ are controlled by Lemma~\ref{lemma.f}; we
thus get~(d) by applying~\eqref{eq.phibound}, choosing
$\bq_1,\bq_2,\bq_3$ % ,\ha q_2$ 
sufficiently large.

Comparing~$\bT$ and~$T$ in $C^0$-norm in the ball $B(0,\abs{z_0}+1)$,
since we can take $\bom$ arbitrarily close to~$\om$ and the $\bq_i$'s
arbitrarily large,
%
% by~(a) and~\eqref{eq.phibound}, we see that $\norm{\bT-T}_{C^0(B(0,R))} \le\eta^2$. 
%
%\margem{David: argument $C^0$ OK?}
%
% Hence, if $\eta$ is sufficiently small, 
%
we can find $\bar{z}_0 \in X$
such that $\abs{ \bar{z}_0-z_0}\leq \eta$ and $\bT^M(\bz_0)=z'$. We
thus take $\bM=M+N$, so that ${\bT}^{\bM}(\bz_0)=\bz$ and
$\abs{{\bT}^{m}(\bz_0)_i} \leq (1+\eta)x_i$ for all
$m \in \big\{M,\ldots,\bM\big\}$.

\smallskip

To finish the proof of (c), % of the corollary 
just observe that  
$\bPhi \in \cU^{\al,L}$ and $\bPhi$
is $\bz$-admissible since $\bx_i \leq x_i/2$ and % $\ha q_2^{-3}<x_2/10$, 
%and if $\eta$ is sufficiently small then 
%
%\margem{Attention, j'ai remplacé $\bq_1^{-3}\leq x_1/10$ par
%  $\bq_1^{-3}\leq \ha x_1/10$ un peu au hasard... Est-ce OK? (David)}
%
$\bq_1^{-3}\leq \ha x_1/10$,
$\bq_3^{-3}\leq x_3/10$
(possibly increasing % $\ha q_2$, 
$\bq_1$ and~$\bq_3$ if necessary).
\end{proof} 

% \bigskip

% David: IL Y A UNE ERREUR DANS CE QUI PRÉC\`EDE, à cause de
% l'apparition de $\ha q_2$ à la place de~$q_2$ à certains endroits, par
% exemple la dernière formule doit être corrigée en quelque chose comme
% %
% \[
% %
% \bPsi = \Phi_{2,1,\bx_2,\bq_2^{-3}} \circ
% %
% \Phi_{1,3,\ha x_1,\bq_1^{-3}} \circ \Phi_{3,2,x_3,\bq_3^{-3}} 
% %
% \circ \Phi_{2,1,x_2,\ha q_2^{-3}} \circ \Phi^{-1}_{2,1,x_2,q_2^{-3}} 
% %
% \circ \Psi,
% %
% \]
% %
% ce qui veut dire que $\Psi^\infty$ ne sera plus exactement un produit
% infini, 
% %
% mais si~\eqref{eqPhipdtinf} est bien correcte, j'espère qu'on pourra
% quand même obtenir la convergence.
% %
% \`A VOIR!!!

\medskip

Clearly, Proposition \ref{prop:iterative} is tailored so that it can
be applied inductively.  The gain obtained when going from $T$ to
$\bT$ is twofold : on the one hand
the orbit of 
the new initial point $\bar{z}_0$ is pushed further close to the
origin, and on the other hand the rotation vector at the origin is
changed to behave increasingly like an non-resonant vector.

\begin{proof}[Proof of Theorem \ref{theo.main}]
Let $\ga>0$.
We pick
\[ \om\exnz=(p_1\exnz/q_1\exnz,p_2\exnz/q_2\exnz,p_3\exnz/q_3\exnz)
\in \Q_+^3 \]
with % $q_2\exnz \ge 2/R$ and
$q_3\exnz \vert q_1\exnz \vert q_2\exnz$,
and $x_1\exnz, x_2\exnz, x_3\exnz>0$ so that $x_2\exnz \ge 1/q_2\exnz$ and
\[
z_0\exnz \defeq ((x_1\exnz,0),(x_2\exnz,0),(x_3\exnz,0)) \in
B(0,R/2).
\]
Let $\Phi\exnz \defeq \Id$ and $M\exnz \defeq 0$.
Define 
\[
T\exnz \defeq \Psi\exnz \circ  S_{\om\exnz}
\quad \text{with} \quad
\Psi\exnz \defeq \Phi_{2,1,x_2\exnz,1/(q_2\exnz)^{3} }\circ \Phi\exnz.
\]
Choosing $q_2\exnz$ sufficiently large, we have $\norm{\Psi\exnz
  -\Id}_{\al,L} \le \min\{\eps/2,\ga/2\}$ by \eqref{ineqPhiijznu}. %{ineqPhiuId}.
The hypotheses of Proposition \ref{prop:iterative} hold for $z\exnz=z_0\exnz$.
 
\medskip

We apply Proposition \ref{prop:iterative} inductively by choosing
inductively a
sequence $(\eta\exn)_{n\ge1}$ such that
\[
\eta\exn \le \min\Big\{
%
%{\blue 
\frac{\eps}{2^{n+1}}, \frac{\ga}{2^{n+1}}, 1/10
% \frac{\eps}{2^{n+2}}, \frac{\ga}{2^{n+2}}, 1/10
%
\Big\},
\qquad
\sum_{k=n+1}^\infty \eta^{(k)} \le \frac{\eta\exn}{\bq_2\exn}
\]
(where $\bq_2\exn$ is determined at the $n$th step of the induction).
We get sequences 
$(\omega\exn)_{n\ge0}$, 
$(z_0\exn)_{n\ge0}$, 
$(z\exn)_{n\ge0}$,
$(T\exn)_{n\ge0}$,
$(M\exn)_{n\ge0}$, with 
\[
z\exn=((x_1\exn,0),(x_2\exn,0),(x_3\exn,0)),
\quad % \ens\text{such that}\ens 
0 < x_i\exnp \le x_i\exn/2
\]
and $T\exn = \Psi\exn \circ S_{\om\exn}$ 
with $\Psi\exn 
= \Phi_{2,1,x_2\exn,1/(q_2\exn)^{3} }\circ \Phi\exn
\in \cU^{\al,L}$,
so that
\begin{multline}   \label{eqestimCauchy}
%
%%% CORRECTION 18 NOV. 2017
%%% \abs{\om\exnp - \om\exn} \le \big( \eta\exnp \big)^3,
%%%
\abs{\om\exnp - \om\exn} \le \eta\exnp,
\quad
\abs{z_0\exnp - z_0\exn} \le \eta\exnp, \\%[1ex]
\quad
\norm{\Psi\exnp - \Psi\exn}_{\al,L} \le \eta\exnp.
\end{multline}
%
% By letting $\eta\exnp$ be very small compared to all the quantities
% that appear at step $n$, we insure that: 
%
%Taking $\eta\exnp$ small enough, w
We also have
\begin{multline} %\beglab
\label{ineqOrbitMj}
\abs{({T\exnp}^{m}(z_0\exnp))_i} \leq 1.01 x_i^{(j)} \\%[1ex]
\ens\text{for all $m \in \{M^{(j)},\ldots,M^{(j+1)}\}$ with $j\leq n$.}
\end{multline} % \edla
In view of~\eqref{eqestimCauchy}, the sequences $(z_0\exn)$, $(\om\exn)$ and $(\Psi\exn)$ are
Cauchy. 
We denote their limits by $z_0^\infty$, $\om^\infty$ and
$\Psi^\infty$. 
Notice that $\norm{\Psi^\infty-\Id}_{\al,L} \le \ga$.

\medskip

We obtain that ${\bf T} \defeq \Psi^\infty \circ S_{\om^\infty}$ 
satisfies $\abs{{\bf T}^m(z_0^{\infty})} \underset{m\to +\infty}{
  \longrightarrow} 0$,
because the ball $B(0,R)$ is a compact subset of~$X$ which contains all the points
${T\exnp}^{m}(z_0\exnp)$ and on which $T\exn \underset{n\to +\infty}{\longrightarrow} {\bf T}$ in the $C^0$
topology,
hence
${T\exnp}^{m}(z_0\exnp) \underset{n\to +\infty}{
  \longrightarrow} {\bf T}^m(z_0^{\infty})$ for each~$m$
and, in~\eqref{ineqOrbitMj}, we can first let~$n$ tend to~$\infty$ and then
use the fact that $x_i^{(j)} \downarrow 0$ and $M^{(j)}\uparrow
\infty$ as $j$ tends to~$\infty$.

\medskip

The orbits of the translation of vector $\om\exn$ on $\T^3$ 
% of length $L^{(j)}$ are $2\eta^{(j)}$-dense for every $j\leq n$ 
being $\eta\exn$-dense and $\bq_2\exn$-periodic,
% (take the $\eps$-density as an open condition so that it still holds
% under small perturbations by continuity);
%
%From 1) above 
we see that ${\om}^\infty$ defines a minimal translation on $\T^3$.
Indeed, given $\th\in\T^3$ and $\eps>0$, we can choose $n,m\in\N$ so that
$\eta\exn\le\eps/2$,
$\dist(m\om\exn-\th,\Z^3) \le \eta\exn$
and $m < \bq_2\exn$.
Then,
\[ \dist(m\om^\infty-\th,\Z^3) \le \eta\exn +
m\abs{\om^\infty-\om\exn}
\le \eta\exn +
\bq_2\exn \sum_{k=n+1}^\infty \eta^{(k)}
\le 2\eta\exn\]
which is $ \le \eps$.
Hence the orbit of~$0$ under the translation of vector $\om^\infty$ is
$\eps$-dense for every~$\eps$,
which entails that ${\om}^\infty$ is non-resonant.

\medskip

The proof of Theorem \ref{theo.main} is thus complete. \end{proof}

%%%%%%%%%%%%%%%%%%%%%%%%%%%%%%%%%%%%%%%%%%%%
%%%%%%%%%%%%%%%%%%%%%%%%%%%%%%%%%%%%%%%%%%%%

%\newpage

\appendix
\section{Gevrey functions, maps and flows}\label{App:Gevrey}
\setcounter{thm}{-1}

%%%%%%%%%%%%%%%%%%%%%%%%%%%%%%%%%%%%%%%%%%%%
%%%%%%%%%%%%%%%%%%%%%%%%%%%%%%%%%%%%%%%%%%%%

\subsection{Gevrey functions and Gevrey maps}

We follow Section~1.1.2 and Appendix~B of \cite{mcwdGev}, with some simplifications
stemming from the fact that here we only need to consider
functions satisfying uniform estimates on the whole of a Euclidean space.

\subsubsection*{The Banach algebra of uniformly Gevrey-$(\al,L)$ functions}
Let $N\ge1$ be integer and $\al\ge1$ and $L>0$ be real. 
We define
\begin{multline*}
\label{eq:defGalL}
G^{\al,L}(\R^N) \defeq \{ f\in C^\infty(\R^N) \mid \norm{f}_{\al,L} <\infty \},
\\ % \quad
\norm{f}_{\al,L} \defeq \sum_{\ell\in\N^{N}}
\frac{L^{\abs{\ell}\al}}{\ell !^\al} \norm{\pa^\ell f}_{C^0(\R^N)}.
\end{multline*}
We have used the standard notations
$\abs{\ell} = \ell_1+\cdots+\ell_{N}$, $\ell! = \ell_1!\ldots\ell_{N}!$,
$\pa^\ell = \pa_{x_1}^{\ell_1}\ldots\pa_{x_N}^{\ell_{N}}$, 
and
\[
\N \defeq \{0,1,2,\ldots\}.
\]

The space $G^{\al,L}(\R^N)$ turns out to be a Banach algebra,
with
\beglab{ineqGevBanAlg}
\norm{fg}_{\al,L} \le \norm{f}_{\al,L} \norm{g}_{\al,L}
\edla
for all $f,g\in G^{\al,L}(\R^N)$,
%
%{\blue 
and there are ``Cauchy-Gevrey inequalities'':
if $0 < L' < L$,
then all the partial derivatives of~$f$ belong to $G^{\al,L'}(\R^N)$ and,
for each $p\in\N$, %}
\begin{equation}	\label{ineqGevCauch}
%
% {\blue 
\sum_{m\in \N^N;\ |m|=p} \norm{\pa^m f}_{\al,L'} \le 
\frac{p!^\al}{(L-L')^{p\al}} \norm{f}_{\al,L}  % }
\end{equation}
(see \cite{hms}).

%%%%%%%%%%%%%%%%%%%%%%%%%%%%%%%%%%%%%%%%%%%%
%%%%%%%%%%%%%%%%%%%%%%%%%%%%%%%%%%%%%%%%%%%%

\subsubsection*{The Banach space of uniformly Gevrey-$(\al,L)$ maps}
Let $N,M\ge1$ be integer and $\al\ge1$ and $L>0$ be real. 
We define
\begin{multline*}
\label{eq:defGalL}
G^{\al,L}(\R^{N},\R^{M}) \defeq 
\{ F%=(F\cc1,\ldots,F\cc M) 
\in C^\infty(\R^N,\R^M) \mid \norm{F}_{\al,L} <\infty \},
\\
\norm{F}_{\al,L} \defeq \norm{F\cc1}_{\al,L} + \cdots + \norm{F\cc M}_{\al,L} .
\end{multline*}
This is a Banach space.
\medskip

When $N=M=2n$, we denote by $\Id+G^{\al,L}(\R^{2n},\R^{2n})$ the set
of all maps of the form $\Psi=\Id+F$ with $F \in G^{\al,L}(\R^{2n},\R^{2n})$.
This is a complete metric space for the distance 
$\dist(\Id+F_1,\Id+F_2) = \norm{F_2-F_1}_{\al,L}$.
We use the notation
\[ \dist(\Psi_1,\Psi_2) = \norm{\Psi_2-\Psi_1}_{\al,L} \]
as well. We then define 
\[ \cU^{\al,L} \subset \Id+G^{\al,L}(\R^{2n},\R^{2n}) \]
as the subset consisting of all Gevrey-$(\al,L)$ symplectic diffeomorphisms of~$\R^{2n}$ which
fix the origin and are $C^\infty$-tangent to~$\Id$ at the origin.
This is a closed subset of the complete metric space
$\Id + G^{\al,L}(\R^{2n},\R^{2n})$.

%%%%%%%%%%%%%%%%%%%%%%%%%%%%%%%%%%%%%%%%%%%%
%%%%%%%%%%%%%%%%%%%%%%%%%%%%%%%%%%%%%%%%%%%%

\subsubsection*{Composition with close-to-identity Gevrey-$(\al,L)$ maps}

Let $N\ge1$ be integer and $\al\ge1$ and $L>0$ be real. We use the
notation
$\NN \defeq \N^N \setminus \{0\}$
and define
\[ % begin{equation}	\label{eqdefcNst}
\cN^*_{\al,L}(f) \defeq 
\sum_{\ell\in\NN} \frac{L^{\abs{\ell}\al}}{\ell!^\al} 
\norm{\pa^\ell f}_{C^0(\R^N)},
\] % end{equation}
so that
$\norm{f}_{\al,L} = \norm{f}_{C^0(\R^N)} + \cN^*_{\al,L}(f)$.

%%%%%%%%%%%%%%%%%%%%%%%%%%%%%%%%%%%%%%%%%%%%

\begin{lemma}  \label{lemGevCompos}
Let $L_1>L$. There exists $\epsc = \epsc(N,\al,L,L_1)$ such that,
for any $f\in G^{\al,L_1}(\R^N)$
and $F=(F\cc1,\ldots,F\cc N)\in G^{\al,L}(\R^N,\R^N)$, if
\[ 
\cN^*_{\al,L}(F\cc1), \ldots, \cN^*_{\al,L}(F\cc N) \le \epsc,
\]
then $f\circ(\Id+F) \in G^{\al,L}(\R^N)$ and 
$\norm{f\circ(\Id+F)}_{\al,L} \le \norm{f}_{\al,L_1}$. % \norm{f}_{\al,L}$.
\end{lemma}

%%%%%%%%%%%%%%%%%%%%%%%%%%%%%%%%%%%%%%%%%%%%

\begin{proof}
Since $L<L_1$, we can pick $\mu>1$ such that $\mu L^\al < L_1^\al$; we
then choose $a>0$ such that $(1+a)^{\al-1}\le\mu$ and set 
$\la \defeq \big( N(1+1/a) \big)^{\al-1}$. We will prove the lemma
with
$\epsc \defeq (L_1^\al - \mu L^\al) / \la$.

Let $f$ and~$F$ be as in the statement,
and $g \defeq f\circ(\Id+F)$. Computing the Taylor expansion of 
$g(x+h) = f(x+h+F(x+h))$ at $h=0$, we get,
for each $k\in \N^N$, $\dfrac{1}{k!}\pa^k g = $
\[
\sum_{\stackrel{\scriptstyle \ell,m,n\in\N^N}{m+n=k}} 
\frac{(\pa^{\ell+n}f)\circ (\Id+F)}{\ell!\, n!}
\sum_{\stackrel{\scriptstyle k^1,\ldots,k^{\abs{\ell}}\in\NN}{k^1+\cdots+k^{\abs{\ell}}=m}}
\frac{\prod\limits_{i=1}^{N} \;
\prod\limits_{\ell_1+\cdots+\ell_{i-1} < p \le \ell_1+\cdots+\ell_i}
\pa^{k^p} F\cc i}{k^1!\cdots k^{\abs{\ell}}!}
\]
with the convention that an empty sum is~$0$ and an empty product is~$1$.
Note that if $\ell=0$, then necessarily $m=0$ and the corresponding
contribution to the sum is $\frac{1}{k!}(\pa^kf)\circ (\Id+F)$,
whereas $\ell\neq0$ implies $m\neq0$ and $k\neq0$.

We have $\norm{g}_{C^0(\R^N)} \le \norm{f}_{C^0(\R^N)}$ and, for each $k\in\NN$,
\[
\frac{1}{k!} \norm{\pa^k g}_{C^0} \le
\frac{1}{k!} \norm{\pa^k f}_{C^0} +
\sum_{\stackrel{\scriptstyle \ell,m,n\in\N^N}{\ell\neq0,\ m+n=k}} 
\frac{\norm{\pa^{\ell+n}f}_{C^0}}{\ell!\, n!}
\sum_{\stackrel{\scriptstyle k^1,\ldots,k^{\abs{\ell}}\in\NN}{k^1+\cdots+k^{\abs{\ell}}=m}}
\frac{P}{k^1!\cdots k^{\abs{\ell}}!}
\]
with 
$P \defeq \prod\limits_{i=1}^{N} \;
\prod\limits_{\ell_1+\cdots+\ell_{i-1} < p \le \ell_1+\cdots+\ell_i}
\norm{\pa^{k^p} F\cc i}_{C^0}$.
Multiplying by $L^{\abs{k}\al}/k!^{\al-1}$ and taking the sum
over~$k$, we get
\beglab{ineqnormgalL}
\norm{g}_{\al,L} \le
\sum_{k\in\N^N} \frac{L^{\abs{k}\al}}{k!^\al} \norm{\pa^k f}_{C^0} +
S
\edla
with
\beglab{eqdefS}
S \defeq 
\sum_{\ell\in\NN,\ m,n\in\N^N}
\frac{L^{\abs{m+n}\al} \norm{\pa^{\ell+n}f}_{C^0}}{\ell!n!(m+n)!^{\al-1}}
\sum_{\stackrel{\scriptstyle k^1,\ldots,k^{\abs{\ell}}\in\NN}{k^1+\cdots+k^{\abs{\ell}}=m}}
\frac{P}{k^1!\cdots k^{\abs{\ell}}!}
\edla
with the same~$P$ as above.

Inequality~(A.7) from~\cite{hms} says that, if $s\ge1$ and
$k^1,\ldots,k^s\in\NN$ with $k^1+\cdots+k^s=m$, then
$k^1! \cdots k^s! \le N^s m! / s!$.
Hence, in each term of the sum~$S$, we can compare
$D \defeq \ell!n!(m+n)!^{\al-1}k^1!\cdots k^{\abs{\ell}}!$
and $\ti D \defeq \ell!n!(\ell+n)!^{\al-1}k^1!^\al\cdots k^{\abs{\ell}}!^\al$:
we have
\begin{multline*}
\frac{\ti D}{D} = 
\Big( \frac{ k^1!\cdots k^{\abs{\ell}}! (\ell+n)! }{ (m+n)! } \Big)^{\al-1}
\le \Big( \frac{ N^{\abs{\ell}} m! (\ell+n)! }{ \abs{\ell}! (m+n)! } \Big)^{\al-1}
\\
\le \Big( \frac{ N^{\abs{\ell}} (\ell+n)! }{ \ell!\, n! } \Big)^{\al-1}
\le \la^{\abs{\ell}} \mu^{\abs{n}},
\end{multline*}
where the last inequality stems from our choice of~$\la$ and~$\mu$,
using
$\frac{ (\ell+n)! }{ \ell!\, n! } \le (1+1/a)^{\abs{\ell}} (1+a)^{\abs{n}}$.
Inserting $\dfrac{1}{D} \le \dfrac{\la^{\abs{\ell}} \mu^{\abs{n}}}{\ti
  D}$ in~\eqref{eqdefS}, we obtain
\[
S \le 
\sum_{\ell\in\NN,\ n\in\N^N}
\frac{L^{\abs{n}\al} \la^{\abs{\ell}} \mu^{\abs{n}} \norm{\pa^{\ell+n}f}_{C^0}}{\ell!n!(\ell+n)!^{\al-1}}
\sum_{k^1,\ldots,k^{\abs{\ell}}\in\NN}
\frac{L^{|k^1+\cdots+k^{\abs{\ell}}|\al} P}{k^1!^\al\cdots k^{\abs{\ell}}!^\al}.
\]
The inner sum over $k^1,\ldots,k^{\abs{\ell}}\in\NN$ coincides with the product
$\cN^*_{\al,L}(F\cc1)^{\ell_1}\cdots \cN^*_{\al,L}(F\cc N)^{\ell_N}$,
which is $\le \epsc^{\abs{\ell}}$ by assumption.
Hence, coming back to~\eqref{ineqnormgalL}, we get
\[
\norm{g}_{\al,L} \le
\sum_{\ell, n\in\N^N}
\frac{(\mu L^\al)^{\abs{n}} (\la\epsc)^{\abs{\ell}} \norm{\pa^{\ell+n}f}_{C^0}}{\ell!n!(\ell+n)!^{\al-1}}
= \sum_{k\in\N^N} 
\frac{(\mu L^\al +\la\epsc)^{\abs{k}} \norm{\pa^k f}_{C^0}}{k!^\al}
\]
(we have used $\mu\ge1$ to absorb the first term of the \rhs
of~\eqref{ineqnormgalL} in the contribution of $\ell=0$).
The conclusion follows from our choice of~$\epsc$.
\end{proof}

%\newpage

%%%%%%%%%%%%%%%%%%%%%%%%%%%%%%%%%%%%%%%%%%%%
%%%%%%%%%%%%%%%%%%%%%%%%%%%%%%%%%%%%%%%%%%%%

\subsection{Estimates for Gevrey flows}
\label{secappGevFlEsti}

% The following lemma is a simple adaptation of ideas from Appendix~B of
% \cite{mcwdGev} (the main ingredient is the stability of Gevrey-$\al$
% maps by composition).

We need some improvements with respect to \cite{hms} and
\cite{mcwdGev} for the estimates of the flow of a small Gevrey
vector field.

\begin{lemma}  \label{lemGevHamFlow}
Suppose $\al\ge1$ and $0<L<L_1$.
%
%\begin{enumerate}[(i)]
%
\smallskip

\noindent (i) %\item
%
%{\blue 
For every integer $N\ge1$, there exists $\epsf=\epsf(N,\al,L,L_1)$
such that, for every vector field $X \in G^{\al,L_1}(\R^N,\R^N)$, if
$\norm{X}_{\al,L_1} \le \epsf$,
% \norm{X\cc1}_{\al,L}, \ldots, \norm{X\cc N}_{\al,L} \le \epsf,
%
then the time-$1$ map $\Phi$ of the flow generated by~$X$ belongs to 
$\Id+G^{\al,L}(\R^N,\R^N)$ and %}
\begla
%
%{\blue 
\norm{\Phi - \Id}_{\al,L} \le \norm{X}_{\al,L_1}. %}
\edla

\noindent (ii) % \item
For every integer $n\ge1$, there exists $\epsH=\epsH(n,\al,L,L_1)$
such that, for every $u\in G^{\al,L_1}(\R^{2n})$,
%
%{\blue 
if $\norm{u}_{\al,L_1} \le \epsH$, %}
then the time-$1$ map~$\Phi^u$ of the Hamiltonian flow generated by~$u$
belongs to $\Id + G^{\al,L}(\R^{2n},\R^{2n})$
and
\beglab{ineqPhiuId}
\norm{\Phi^u - \Id}_{\al,L} \le %{\blue 
2^\al(L_1-L)^{-\al} %} 
\norm{u}_{\al,L_1}.
\edla
%
%\end{enumerate}
%
\end{lemma}

% DONNER LA BONNE CONSTANTE 
% %
% $C = \tfrac{n!^\al}{(L_1-L)^\al}$ 
% %
% ou qqch comme cela...

\medskip

Building upon the previous result, we get

\begin{lemma}   \label{lemComposFlow}
Suppose $\al\ge1$ and $0<L<L_1$.
Then there exist $C=C(n,\al,L,L_1)$ and $\eps=\eps(n,\al,L,L_1)$ such that,
if $r\ge1$, $u_1,\ldots,u_r\in G^{\al,L_1}(\R^{2n})$, 
$\Psi\in \ID + G^{\al,L}(\R^{2n},\R^{2n})$ and
\beglab{ineqHypPsiu}
\norm{\Psi-\Id}_{\al,L} + C 
\big( 
%{\blue 
\norm{u_1}_{\al,L_1}+\cdots+\norm{u_{r}}_{\al,L_1} %} 
\big)
\le\eps,
\edla
then
\beglab{eq.phibound}
\norm{\Phi^{u_r}\circ \cdots \circ \Phi^{u_1} \circ \Psi - \Psi}_{\al,L} \le 
C \big(\norm{u_1}_{\al,L_1}+\cdots+\norm{u_r}_{\al,L_1}\big)
\edla 
%}
%
(with the same notation as in Lemma~\ref{lemGevHamFlow}(ii) for
the~$\Phi^{u_i}$'s).
\end{lemma}

%%%%%%%%%%%%%%%%%%%%%%%%%%%%%%%%%%%%%%%%%%%%
%%%%%%%%%%%%%%%%%%%%%%%%%%%%%%%%%%%%%%%%%%%%

\subsubsection*{Proof of Lemma~\ref{lemGevHamFlow}}

%%%%%%%%%%%%%%%%%%%%%%%%%%%%%%%%%%%%%%%%%%%%
%%%%%%%%%%%%%%%%%%%%%%%%%%%%%%%%%%%%%%%%%%%%

(i) Let us pick $L' \in (L,L_1)$. We will prove the statement with
$\epsf \defeq \epsc(N,\al,L,L')$
(notation from Lemma~\ref{lemGevCompos}).

Let~$X$ be as in the statement. We write the restriction of its flow
to the time-interval $[0,1]$ in the form
$\Phi(t) = \ID + \xi(t)$,
with $t \in [0,1] \mapsto \xi(t)\in C^\infty(\R^N,\R^N)$ characterised by
\[
\xi(t) = \int_0^t X\circ\big(\ID + \xi(\tau)\big)\,\dd\tau
\quad\text{for all $t\in[0,1]$}.
\]
We will show that~$\xi$ belongs to
$\cB \defeq \{\, \psi \in C^0\big( [0,1], G^{\al,L}(\R^N,\R^N) \big)
\mid \norm{\psi} \le \norm{X}_{\al,L_1} \,\}$, 
which is a closed ball in a Banach space.

Lemma~\ref{lemGevCompos} shows that the formula
$\cF(\psi)(t) \defeq \int_0^t X\circ\big(\ID +
\psi(\tau)\big)\,\dd\tau$
defines a map from~$\cB$ to~$\cB$.
Moreover, if $\psi,\psi^*\in\cB$ satisfy 
\[
\norm{\psi^*(t)-\psi(t)}_{\al,L} \le A(t)
\ens\text{for all $t\in[0,1]$},
\]
where $t\in[0,1] \mapsto A(t)$ is continuous, then 
for each~$t$ and~$i$, % $t\in[0,1]$ and $i=1,\ldots,N$,
\begin{multline*}
\cF(\psi^*)(t)\cc i-\cF(\psi)(t)\cc i = 
\int_0^t \dd \tau \sum_{j=1}^N \int_0^1 \dd\th \\
\pa_{x_j}X\cc i\circ\big( \ID + (1-\th)\psi(\tau) + \th\psi^*(\tau) \big)
\big( \psi^*(\tau)\cc j-\psi(\tau)\cc j \big),
\end{multline*}
whence
\[
\norm{ \cF(\psi^*)(t)-\cF(\psi)(t) }_{\al,L} \le K \int_0^t
A(\tau)\,\dd\tau
\ens\text{with $K \defeq \max_{i,j} \norm{\pa_{x_j}X\cc i}_{\al,L'}$} 
\]
(we have $K < \infty$ by~\eqref{ineqGevCauch} and we have used
Lemma~\ref{lemGevCompos} and~\eqref{ineqGevBanAlg}).
Iterating this, we get 
\[
\norm{\cF^p(\psi^*) - \cF^p(\psi)} \le \frac{K^p}{p!} \norm{\psi^*-\psi}
\ens\text{for all $p\in\N$},
\]
which shows that~$\cF^p$ is a contraction for~$p$ large enough. The
map~$\cF$ thus has a unique fixed point in~$\cB$, and this fixed point
is~$\xi$.

\medskip

\noindent
(ii) Let $L' \defeq (L+L_1)/2$.
For any $u\in G^{\al,L_1}(\R^{2n})$, inequality~\eqref{ineqGevCauch}
with $p=1$ reads
\[
\sum_{m\in \N^{2n};\ |m|=1} \norm{\pa^m u}_{\al,L'} \le 
(L_1-L')^{-\al} \norm{u}_{\al,L_1}.
\]
The \lhs is precisely the $(\al,L')$-Gevrey norm of the Hamiltonian
vector field generated by~$u$.
Therefore, point~(i) shows that the conclusion holds with $\epsH = (L_1-L')^\al \epsf(2n,\al,L,L')$.

%%%%%%%%%%%%%%%%%%%%%%%%%%%%%%%%%%%%%%%%%%%%
%%%%%%%%%%%%%%%%%%%%%%%%%%%%%%%%%%%%%%%%%%%%

\subsubsection*{Proof of Lemma~\ref{lemComposFlow}}

%%%%%%%%%%%%%%%%%%%%%%%%%%%%%%%%%%%%%%%%%%%%
%%%%%%%%%%%%%%%%%%%%%%%%%%%%%%%%%%%%%%%%%%%%

Let us pick $L' \in(L,L_1)$. We will show the statement with
\[ C\defeq 2^\al (L_1-L')^{-\al}, \quad
\eps \defeq \min\big\{
\epsc(2n,\al,L,L'), C \epsH(n,\al,L',L_1) \big\} \]
by induction on~$r$.

The induction is tautologically initialized for $r=0$. Let us take
$r\ge1$ and assume that the statement holds at rank $r-1$.
Given $u_1,\ldots,u_r\in G^{\al,L_1}(\R^{2n})$ and
$\Psi\in \ID + G^{\al,L}(\R^{2n},\R^{2n})$ 
satisfying~\eqref{ineqHypPsiu}, we set
$\chi \defeq \Phi^{u_{r-1}}\circ \cdots \circ \Phi^{u_1} \circ \Psi$,
which satisfies 
\[
\norm{\chi-\Psi}_{\al,L} \le 
C \big(\norm{u_1}_{\al,L_1}+\cdots+\norm{u_{r-1}}_{\al,L_1}\big)
\]
by the induction hypothesis,
and observe that we also have
\[ \norm{\Phi^{u_r}-\ID}_{\al,L'} \le C \norm{u_r}_{\al,L_1} \]
since $\norm{u_r}_{\al,L_1}\le\epsH(n,\al,L',L_1)$.
Now
\begin{align*}
\norm{\Phi^{u_r}\circ \cdots \circ \Phi^{u_1} \circ \Psi - \Psi}_{\al,L} &\le 
\norm{ (\Phi^{u_r}-\ID)\circ\chi }_{\al,L} + \norm{\chi-\Psi}_{\al,L}
\\[1ex] 
&\le \norm{ \Phi^{u_r}-\ID }_{\al,L'} + \norm{\chi-\Psi}_{\al,L}
\end{align*}
since $\norm{\chi-\Id}_{\al,L} \le 
\norm{\Psi-\Id}_{\al,L} + \norm{\chi-\Psi}_{\al,L}
\le \norm{\Psi-\Id}_{\al,L} + 
C \big(\norm{u_1}_{\al,L_1}+\cdots+\norm{u_{r-1}}_{\al,L_1}\big)
\le \epsc(2n,\al,L,L')$
and we are done.

%%%%%%%%%%%%%%%%%%%%%%%%%%%%%%%%%%%%%%%%%%%%
%%%%%%%%%%%%%%%%%%%%%%%%%%%%%%%%%%%%%%%%%%%%

\end{document}